\documentclass[a4paper, 11pt]{article}

\usepackage{indentfirst}
\usepackage{amsmath}
\usepackage{amsthm}
\usepackage{amsfonts}
\usepackage{bbm}
\usepackage{accents}
\newcommand{\ubar}[1]{\underaccent{\bar}{#1}}

\newcommand{\me}{\mathbb{E}}
\newcommand{\mr}{\mathbb{R}}

\newcommand{\mn}{\mathbb{N}}
\newcommand{\mmp}{\mathbb{P}}

\DeclareMathOperator{\1}{\mathbbm{1}}
\newcommand{\fdc}{\overset{\mathrm{f.d.}}{\Longrightarrow}}

\newcommand{\eee}{{\rm e}}

\newtheorem{thm}{Theorem}[section]
\newtheorem{lemma}[thm]{Lemma}

\newtheorem{assertion}[thm]{Proposition}
\theoremstyle{definition}

\theoremstyle{remark}
\newtheorem{rem}[thm]{Remark}

\begin{document}

\title{Limit theorems for globally perturbed random walks}\date{}
\author{Alexander Iksanov\footnote{Faculty of Computer Science and Cybernetics, Taras Shevchenko National University of Kyiv, Ukraine; e-mail address:
iksan@univ.kiev.ua} \ \ and \ \ Oleh Kondratenko\footnote{Faculty of Computer Science and Cybernetics, Taras Shevchenko National University of Kyiv, Ukraine; e-mail address:
kondratolegua@gmail.com} }
\maketitle
\begin{abstract}
\noindent Let $(\xi_1, \eta_1)$, $(\xi_2, \eta_2),\ldots$ be independent copies of an $\mathbb{R}^2$-valued random vector $(\xi, \eta)$ with arbitrarily dependent components. Put $T_n:= \xi_1+\ldots+\xi_{n-1} + \eta_n $ for $n\in\mathbb{N}$ and
define $\tau(t) := \inf\{n\geq 1: T_n>t\}$ the first passage time into $(t,\infty)$, $N(t) :=\sum_{n\geq 1}\1_{\{T_n\leq t\}}$ the number of visits to $(-\infty, t]$ and $\rho(t) ~:=~ \sup\{n\geq 1: T_n \leq t\}$ the associated last exit time for $t\in\mathbb{R}$. The standing assumption of the paper is $\me[\xi]\in (0,\infty)$. We prove a weak law of large numbers for $\tau(t)$ and strong laws of large numbers for $\tau(t)$, $N(t)$ and $\rho(t)$. The strong law of large numbers for $\tau(t)$ holds if, and only if, $\me[\eta^+]<\infty$. In the complementary situation $\me[\eta^+]=\infty$ we prove functional limit theorems in the Skorokhod space for $(\tau(ut))_{u\geq 0}$, properly normalized without centering. Also, we provide sufficient conditions under which finite dimensional distributions of $(\tau(ut))_{u\geq 0}$, $(N(ut))_{u\geq 0}$ and $(\rho(ut))_{u\geq 0}$, properly normalized and centered, converge weakly as $t\to\infty$ to those of a Brownian motion. Quite unexpectedly, the centering needed for $(N(ut))$ takes in general a more complicated form than the centering $ut/\me[\xi]$ needed for $(\tau(ut))$ and $(\rho(ut))$. Finally, we prove a functional limit theorem in the Skorokhod space for $(N(ut))$ under optimal moment conditions.
\end{abstract}


\hphantom{2000 Mathematics Subject Classification: } Secondary: 60G50

\section{Introduction}
Let $(\xi_1, \eta_1)$, $(\xi_2, \eta_2),\ldots$ be independent copies of an $\mr^2$-valued random vector $(\xi, \eta)$ with arbitrarily dependent components. Denote by $(S_n)_{n\geq 0}$ the zero-delayed standard random walk with increments $\xi_n$ for $n\in\mn:=\{1,2,\ldots\}$, that is, $S_0:=0$ and $S_n:=\xi_1+\ldots+\xi_n$ for $n\in\mn$. Put
\begin{equation*}
T_n:= S_{n-1}+ \eta_n,\quad n\in\mn.
\end{equation*}
The sequence $T:=(T_n)_{n\geq 1}$ is called {\it globally perturbed random walk}. Many results concerning $T$ accumulated up to 2016 can be found in the book \cite{Iksanov:2016}. An incomplete list of more recent publications in which the sequence $T$ is either the main object of investigation or plays an important role includes \cite{Basrak etal:2022, Bohdanskyi etal:2024, Bohun etal:2022, Iksanov+Pilipenko+Samoilenko:2017, Iksanov+Rashytov+Samoilenko:2023, Pitman+Tang:2019}. Under the assumptions $\me [\xi]\in (0,\infty)$ and $\me [|\eta|]<\infty$ the sequence $T$ is a particular instance of {\it perturbed random walks} investigated in Section 6 of the book \cite{Gut:2009}. More details will be given in Remarks \ref{rem:gut} and \ref{rem:cohn}.

For $t\in\mr$, define the {\it first passage time} into $(t,\infty)$
\begin{equation*}
\tau(t) := \inf\{n\geq 1: T_n>t\},
\end{equation*}
the {\it number of visits} to $(-\infty, t]$
\begin{equation*}
N(t) :=\sum_{n\geq 1}\1_{\{T_n\leq t\}}
\end{equation*}
and the associated {\it last exit time}
\begin{equation*}
\rho(t) ~:=~ \sup\{n\geq 1: T_n \leq t\}
\end{equation*}
with the usual conventions that $\sup \oslash = 0$ and $\inf \oslash = \infty$. Plainly, for each $t\in\mr$,
\begin{equation}\label{eq:ineq}
\tau(t)-1\leq N(t)\leq \rho(t)\quad \text{a.s.}
\end{equation}
We mark with the $\ast$ the corresponding quantities for $(S_n)_{n\geq 0}$, that is, for $t\in\mr$, $\tau^\ast(t) := \inf\{n\geq 0: S_n>t\}$, $N^\ast(t) :=\sum_{n\geq 0}\1_{\{S_n\leq t\}}$ and $\rho^\ast(t):=\sup\{n\geq 0: S_n \leq t\}$.

We use the standard notation $x^+=\max(x,0)$ and $x^-=\max(-x,0)$ for $x\in\mr$. We proceed by giving a couple of results which can be lifted from the existing literature. Proposition \ref{assert:nrho} follows from Theorem 1.4.1 and Remark 1.2.3 in \cite{Iksanov:2016}, see also Theorem 2.4 in \cite{Alsmeyer+Iksanov+Meiners:2015}.
\begin{assertion}\label{assert:nrho}
Suppose $\mu:=\me[\xi]\in (0,\infty)$. The following assertions are equivalent:

\noindent (a) $\lim_{n\to\infty}T_n= +\infty$ {\rm a.s.};

\noindent (b) $\me [\eta^-]<\infty$;

\noindent (c) $N(t)<\infty$ {\rm a.s.}\ for some/all $t\in\mr$;

\noindent (d) $\rho(t)<\infty$ {\rm a.s.}\ for some/all $t\in\mr$.
\end{assertion}

Proposition \ref{assert:nrho} states that the conditions $\mu\in (0,\infty)$ and $\me [\eta^-]<\infty$ entail $\lim_{n\to\infty}T_n=+\infty$ a.s. According to Theorem 1.2.1 in \cite{Iksanov:2016}, the conditions $\mu\in (0,\infty)$ and $\me [\eta^-]=\infty$ ensure that $\liminf_{n\to\infty}T_n=-\infty$ and $\limsup_{n\to\infty}T_n=+\infty$ a.s.
Lemma \ref{lem:div} is an immediate consequence of these observations.
\begin{lemma}\label{lem:div}
Suppose $\mu=\me[\xi]\in (0,\infty)$. Then $\lim_{n\to\infty}\max_{1\leq j\leq n}\,T_j=+\infty$ {\rm a.s.}
\end{lemma}

The present work was partly motivated by Lemma 4.2 in \cite{Basrak etal:2022}. The cited result provides a strong law of large numbers for $\tau(t)$ which corresponds to a particular globally perturbed random walk. Our purpose is to prove weak and strong laws of large numbers for $\tau(t)$, strong laws of large numbers for $N(t)$ and $\rho(t)$ and distributional limit theorems for $\tau(t)$, $N(t)$ and $\rho(t)$. Our main results are stated in Section \ref{sec:main}.

\section{Main results}\label{sec:main}

\subsection{Weak and strong laws of large numbers}

In Theorem \ref{weak} we formulate a weak law of large numbers for $\tau(t)$.
\begin{thm}\label{weak}
Suppose $\mu=\me [\xi]\in (0,\infty)$. The following assertions are equivalent:

\noindent ({\rm W1}) $\lim_{t\to\infty}t^{-1} \tau(t)=\mu^{-1}$ in probability;

\noindent ({\rm W2}) $\lim_{n\to\infty}n^{-1}\max_{1\leq k\leq n}\, T_k=\mu$ in probability;

\noindent ({\rm W3}) $\lim_{t\to\infty}t\mmp\{\eta>t\}=0$.
\end{thm}

In Theorems \ref{v_t} and \ref{N_t} we formulate strong laws of large numbers for $\tau(t)$, $N(t)$ and $\rho(t)$.
\begin{thm}\label{v_t}
Suppose $\mu=\me [\xi]\in (0,\infty)$. The following assertions are equivalent:

\noindent ({\rm S1}) $\lim_{t\to\infty}t^{-1} \tau(t)=\mu^{-1}$ {\rm a.s.};

\noindent ({\rm S2}) $\lim_{n\to\infty}n^{-1}\max_{1\leq k\leq n}\, T_k=\mu$ {\rm a.s.};

\noindent ({\rm S3}) $\me[\eta^+]<\infty$.
\end{thm}
\begin{rem}\label{rem:gut}
Section 6 of the book \cite{Gut:2009} is concerned with perturbed random walks, which are defined by $$Z_n:=S_{n-1}+\theta_n,\quad n\in\mn.$$ Here, $(S_n)_{n\geq 0}$ is a zero-delayed standard random walk with $\mu=\me [\xi]\in (0,\infty)$ and $(\theta_j)_{j\geq 1}$ is a random sequence arbitrarily dependent on $(S_n)_{n\geq 0}$ and satisfying $\lim_{n\to\infty}n^{-1}\theta_n=0$ a.s. Put $\tau_Z(t):=\inf\{n\geq 1: Z_n>t\}$ for $t\geq 0$. A specialization of Theorem 2.1 on p.~179 of \cite{Gut:2009} reads $\lim_{t\to\infty} t^{-1}\tau_Z(t)=\mu^{-1}$ a.s.

By the Borel-Cantelli lemma, $\lim_{n\to\infty}n^{-1}\eta_n=0$ a.s. if, and only if, $\me [|\eta|]<\infty$. Thus, under the assumptions $\mu\in (0,\infty)$ and $\me [|\eta|]<\infty$, the sequence $T$ is a particular instance of $(Z_n)_{n\geq 1}$. In particular, Theorem 2.1 in \cite{Gut:2009} then entails $\lim_{t\to\infty}t^{-1}\tau(t)=\mu^{-1}$ a.s. This is a weak version of the implication $({\rm S}3)\Rightarrow ({\rm S}1)$ of Theorem \ref{v_t}.
\end{rem}
\begin{rem}\label{rem:1}
One may wonder what is the asymptotic behavior of $t^{-1}\tau(t)$ under the conditions $\mu\in (0,\infty)$ and $\me [\eta^+]=\infty$? This is investigated in Theorems \ref{thm1} and \ref{thm2} below. Of particular interest is Theorem \ref{thm1}, in which a slight departure from $\me [\eta^+]<\infty$ is addressed. It turns out that $t^{-1}\tau(t)$ then converges in distribution as $t\to\infty$ to a random variable which is smaller than $\mu^{-1}$ a.s.
\end{rem}
\begin{rem}
Lemma 4.2 in \cite{Basrak etal:2022} proves the implication $({\rm S3})\Rightarrow ({\rm S1})$ of Theorem \ref{v_t} for a globally perturbed random walk in which the vector $(\xi, \eta)$ has a specific distribution. However, its proof works equally well whenever $\mu\in (0,\infty)$,  $\me[\eta^+]<\infty$ and beyond that, the distribution of $(\xi, \eta)$ is arbitrary. The implication $({\rm S1})\Rightarrow ({\rm S3})$ of Theorem \ref{v_t} seems to be new.
\end{rem}

Theorem \ref{N_t} given next and Theorem \ref{v_t} reveal a remarkable difference between the first order asymptotic behavior of $\tau(t)$ and that of $N(t)$ and $\rho(t)$. The former depends heavily upon the right distribution tail of $\eta$, whereas the latter does not depend on it at all.
\begin{thm}\label{N_t}
Suppose $\mu=\me [\xi]\in (0,\infty)$. If, for some $t\in\mr$, $N(t)<\infty$ a.s.\ or $\rho(t)<\infty$ a.s., then
\begin{equation*}
\lim_{t\to\infty}\frac{N(t)}{t}=\lim_{t\to\infty}\frac{\rho(t)}{t}=\frac{1}{\mu}\quad \text{{\rm a.s.}}
\end{equation*}
\end{thm}
Theorems \ref{v_t} and \ref{N_t} are generalizations of the following previously known strong laws of large numbers for $\tau^\ast$, $N^\ast$ and $\rho^\ast$, see Theorem 4.1 on p.~88 and formulae (4.7) and (4.8) on p.~90 in \cite{Gut:2009}.
\begin{assertion}\label{prop:slln}
Suppose $\mu=\me[\xi]\in (0,\infty)$. Then $$\lim_{t\to\infty}\frac{\tau^\ast(t)}{t}=\lim_{t\to\infty}\frac{N^\ast(t)}{t}=\lim_{t\to\infty}\frac{\rho^\ast(t)}{t}
=\frac{1}{\mu}\quad\text{{\rm a.s.}}$$
\end{assertion}
We do not provide a new proof of Proposition \ref{prop:slln}. Rather, this proposition is an important ingredient of our proofs of Theorems \ref{v_t} and \ref{N_t}.

\subsection{Functional limit theorems}
For $a,b>0$, let  $(t^{(a, b)}_k, j^{(a, b)}_k)$ be the atoms of a Poisson random measure $N^{(a,b)}$ on $[0, \infty)\times(0, \infty)$ with mean measure $\mathbb{LEB}\times\mu_{a, b}$, where $\mathbb{LEB}$ is Lebesgue measure on $[0, \infty)$ and $\mu_{a, b}$ is a measure on  $(0, \infty]$ defined by
\begin{equation*}
\mu_{a, b}((x, \infty]) = ax^{-b},\quad x>0.
\end{equation*}
Denote by $D:=D[0,\infty)$ the Skorokhod space, that is, the set of c\`{a}dl\`{a}g functions defined on $[0,\infty)$. We shall use the $J_1$- and $M_1$-topologies, which are standard topologies on $D$. Comprehensive information on the $J_1$-topology and the $M_1$-topology can be found in \cite{Billingsley:1968, Ethier+Kurtz:2005} and \cite{Whitt:2002}, respectively. We write $\Longrightarrow$ to denote weak convergence in a function space.
\begin{thm}\label{thm1}
Suppose $\mu=\me[\xi]\in(-\infty, \infty)$ and $\mmp\{\eta>t\}\sim c/t$ as $t\to\infty$ for some $c>0$. Then
\begin{equation*}
\Big(\frac{\tau(ut)}{t}\Big)_{u\geq 0}~ \Longrightarrow~\big(\inf\{z\geq 0: \max_{k:\,t_k^{(c,\, 1)}\leq z}(\mu t_k^{(c,\, 1)} + j_k^{(c,\, 1)}) > u\}\big)_{u\geq 0}=:(X(u))_{u\geq 0},\quad t\to\infty
\end{equation*}
in the $M_1$-topology on $D$. The one-dimensional distributions of the limit process are given by
\begin{equation}\label{eq:marginal}
\mmp\{X(u)\leq y\} =
\begin{cases}
1-\big(\frac{u-\mu y}{u}\big)^{c/\mu}\1_{[0,\,u/\mu]}(y), & \text{if}\quad \mu > 0,\\
1-\big(\frac{u}{u + |\mu|y}\big)^{c/|\mu|}, & \text{if}\quad \mu < 0, \\
1-\eee^{-cy/u}, & \text{if}\quad \mu = 0
\end{cases}
\end{equation}
for $y\geq 0$ and $\mmp\{X(u)\leq y\}=0$ for $y<0$.
\end{thm}
\begin{rem}
Let $\theta (a,b)$ be a random variable having a beta distribution with positive parameters $a$ and $b$, that is, $$\mmp\{\theta(a,b)\in {\rm d}x\}=({\rm B}(a,b))^{-1}x^{a-1}(1-x)^{b-1}\1_{(0,1)}(x){\rm d}x,$$ where $B$ is the Euler beta function. In the case $\mu>0$, $X(u)$ has the same distribution as $\mu^{-1}u\theta(1,\mu^{-1}c)$. In particular, $X(1)<\mu^{-1}$ a.s., which justifies the claim made in Remark \ref{rem:1}. In the case $\mu<0$, the distribution of $X(u)$ is Pareto-like. In the case $\mu=0$, the process $(X(u))_{u\geq 0}$ is the inverse of an extremal process. It is known (see Proposition 4.8 on p.~183 in \cite{Resnick:1987}) that $(X(u))_{u\geq 0}$ has independent, but not stationary, increments and that its marginal distributions are exponential of mean $u/c$. Of course, the latter is confirmed by \eqref{eq:marginal}.
\end{rem}
\begin{thm}\label{thm2}
Suppose $\mu=\me[\xi]\in(-\infty, \infty)$ and $\mmp\{\eta>x\}\sim x^{-\alpha}\ell(x)$ as $x\to\infty$ for some $\alpha\in (0,1)$ and some $\ell$ slowly varying at $\infty$. Then
\begin{equation}\label{eq:inverseextremal}
\big(\mmp\{\eta>t\}\tau(ut)\big)_{u\geq 0}~\Longrightarrow~ \big(\inf\{z\geq 0: \max_{k:\, t_k^{(1,\,\alpha)}\leq z} j_k^{(1,\,\alpha)}>u\}\big)_{u\geq 0}=:(Y(u))_{u\geq 0},\quad t\to\infty
\end{equation}
in the $M_1$-topology on $D$. The one-dimensional distributions of the limit process, which is the inverse of an extremal process, are exponential and given by
\begin{equation*}
\mmp\{Y(u)>x\}= \eee^{-u^{-\alpha}x},\quad x,u>0.
\end{equation*}
\end{thm}
\begin{rem}\label{rem:resemblance}

Theorems \ref{thm1} and \ref{thm2} will be proved by an `inversion' of the known functional limit theorems for the maximum $\max_{1\leq j\leq n}\,T_j$, properly scaled. Assume now that $\lim_{n\to\infty}\,T_n=-\infty$ a.s., which particularly implies that $\sup_{j\geq 1}\,T_j$ is a.s.\ finite. Then it makes sense to investigate the distributional tail behavior of $\sup_{j\geq 1}\,T_j$. Actually, this used to be a rather popular subject of research in the recent past \cite{Araman+Glynn:2006, Hao+Tang+Wei:2009, Palmowski+Zwart:2007, Palmowski+Zwart:2010, Robert:2005}.

In the setting of Theorem \ref{thm1}, a major contribution to the limit is made by large values of the perturbations $(\eta_k)_{k\geq 1}$ and also by the standard random walk $(S_n)_{n\geq 0}$ via a law of large numbers. This bears a resemblance to Theorem 4 in \cite{Araman+Glynn:2006} and Theorem 1 in \cite{Palmowski+Zwart:2007}, see also \cite{Palmowski+Zwart:2010}.
\end{rem}

Theorem \ref{thm4} is a result of different flavor. It quantifies the rate of convergence in laws of large numbers for $\tau(t)$, $N(t)$ and $\rho(t)$. We write $\fdc$ to denote weak convergence of finite-dimensional distributions.
\begin{thm}\label{thm4}
Suppose $\mu=\me[\xi]\in(0, \infty)$ and $\sigma^2:={\rm Var}\,[\xi]\in (0,\infty)$.

\noindent If $\me[\eta^+]<\infty$, then
\begin{equation}\label{eq:tau}
\Big(\frac{\tau(ut)-\mu^{-1}ut}{(\sigma^2\mu^{-3}t)^{1/2}}\Big)_{u\geq0}~\fdc~ \big(B(u)\big)_{u\geq0},\quad t\to\infty,
\end{equation}
where $(B(u))_{u\geq0}$ is a standard Brownian motion.

\noindent If $\me[\eta^-]<\infty$, then
\begin{equation}\label{eq:rho}
\Big(\frac{\rho(ut)-\mu^{-1}ut}{(\sigma^2\mu^{-3}t)^{1/2}}\Big)_{u\geq0}~\fdc~ \big(B(u)\big)_{u\geq0},\quad t\to\infty.
\end{equation}

\noindent If $\me [\eta]\in (-\infty,\infty)$, then
\begin{multline}
		\Big(\Big(\frac{\tau(ut)-\mu^{-1}ut}{(\sigma^2\mu^{-3}t)^{1/2}}\Big)_{u\geq0}, \Big(\frac{N(ut)-\mu^{-1}ut}{(\sigma^2\mu^{-3}t)^{1/2}}\Big)_{u\geq 0}, \Big(\frac{\rho(ut)-\mu^{-1}ut}{(\sigma^2\mu^{-3}t)^{1/2}}\Big)_{u\geq0}\Big)\\~\fdc~
		\big(\big(B(u)\big)_{u\geq0}, \big(B(u)\big)_{u\geq0}, \big(B(u)\big)_{u\geq 0}\big),\quad t\to\infty.\label{eq:mix}
\end{multline}
\end{thm}
\begin{rem}\label{rem:cohn}
Let $(Z_n)_{n\geq 1}$ be a perturbed random walk as in Remark \ref{rem:gut}. Assume that $\sigma^2={\rm Var}[\xi]\in (0,\infty)$ and $\lim_{n\to\infty} n^{-1/2}\max_{1\leq k\leq n}\,|\theta_k|=0$ in probability. According to Theorem 1 in \cite{Larrson:2000}, $$\Big(\frac{\tau_Z(un)-\mu^{-1}un}{(\sigma^2\mu^{-3}n)^{1/2}}\Big)_{u\geq0}~\fdc~ \big(B(u)\big)_{u\geq0},\quad n\to\infty$$ in the $J_1$-topology on $D$, where $(B(u))_{u\geq 0}$ is a standard Brownian motion.

Observe that $\lim_{n\to\infty} n^{-1/2}\max_{1\leq k\leq n}\,|\eta_k|=0$ in probability if, and only if, $\lim_{t\to\infty}t^2\mmp\{|\eta|>t\}=0$. Hence, both \eqref{eq:tau} and its functional $J_1$-version hold true under the assumptions $\sigma^2\in (0,\infty)$ and $\lim_{t\to\infty}t^2\mmp\{|\eta|>t\}=0$ (since $u\mapsto \tau(u)$ is a.s.\ nondecreasing, a passage from a discrete parameter $n$ to a continuous parameter $t$ is trivial). We believe that the functional $J_1$-versions of  \eqref{eq:tau} and \eqref{eq:rho} hold true under the assumptions of Theorem \ref{thm4} ensuring weak convergence of the finite-dimensional distributions.
At the moment, a proof of this claim is beyond our reach.
\end{rem}

The limit theorem for $N(t)$ is only given under rather restrictive assumption $\me [\eta]\in (-\infty,+\infty)$. Our final result demonstrates that when stated under the optimal assumption $\me[\eta^-]<\infty$ the limit theorem for $N(t)$ is more interesting than those for $\tau(t)$ and $\rho(t)$. Its feature is a two-term centering. We stress that Theorem \ref{thm5} is a result on weak convergence on $D$ rather than weak convergence of finite dimensional distributions.
\begin{thm}\label{thm5}
Suppose $\mu=\me[\xi]\in(0, \infty)$, $\sigma^2={\rm Var}\,[\xi]\in (0,\infty)$ and $\me[\eta^-]<\infty$. Then
\begin{equation}\label{eq:N}
\Big(\frac{N(ut)-\mu^{-1}ut+\mu^{-1}\int_0^{ut}\mmp\{\eta>y\}{\rm d}y}{(\sigma^2\mu^{-3}t)^{1/2}}\Big)_{u\geq0}~\Longrightarrow~ \big(B(u)\big)_{u\geq0},\quad t\to\infty
\end{equation}
in the $J_1$-topology on $D$, where $(B(u))_{u\geq0}$ is a standard Brownian motion.
\end{thm}
\begin{rem}
In the case of nonnegative $\xi$ and $\eta$ relation \eqref{eq:N} was proved in Theorem 3.2 of \cite{Alsmeyer+Iksanov+Marynych:2017} under the extra assumption $\me [\eta^a]<\infty$ for some $a>0$. Theorem \ref{thm5} demonstrates that the latter assumption is not needed.
\end{rem}
\begin{rem}
If $\me [(\eta^+)^{1/2}]<\infty$, then $\lim_{t\to\infty}t^{-1/2}\int_0^t \mmp\{\eta>y\}{\rm d}y=0$. Thus, relation \eqref{eq:N} simplifies in this case to
\begin{equation*}
\Big(\frac{N(ut)-\mu^{-1}ut}{(\sigma^2\mu^{-3}t)^{1/2}}\Big)_{u\geq0}~\Longrightarrow~ \big(B(u)\big)_{u\geq0},\quad t\to\infty
\end{equation*}
in the $J_1$-topology on $D$. If $\me [(\eta^+)^{1/2}]=\infty$, then the two-term centering is inevitable.
\end{rem}

\section{Proofs of the laws of large numbers}

\subsection{Proof of Theorem \ref{weak}}\label{proof_v_t_weak}
We start by proving the equivalence of (W2) and (W3).

\noindent $({\rm W}2)\Rightarrow ({\rm W}3)$. Fix any $\varepsilon\in (0,\mu)$. Since $\lim_{n\to\infty}(S_n -(\mu - \varepsilon)n) = \infty$ a.s., we conclude that the random variable $N:=\sup\{k\geq 0: S_k\leq(\mu-\varepsilon)k\}$
is a.s.\ finite. Let $(a_n)_{n\geq 1}$ be a sequence of
nonnegative integers satisfying $\lim_{n\to\infty}a_n = \infty$ and $a_n=o(n)$, for instance, $a_n =\lfloor \log n\rfloor$ for $n\in\mn$. Here, $\lfloor x\rfloor$ denotes the integer part of $x$. The following inequalities hold a.s.\ on the event $\{N+1\leq a_n\}$
\begin{multline*}
\max_{0\leq k \leq n}{(S_k + \eta_{k + 1})} \geq \max_{N + 1 \leq k \leq n}{(S_k + \eta_{k + 1})} \\
\geq \max_{N + 1 \leq k \leq n}{((\mu - \varepsilon)k + \eta_{k + 1})} \geq \max_{a_n \leq k \leq n}{((\mu - \varepsilon)k + \eta_{k + 1})}.
\end{multline*}
Using this we obtain
\begin{multline*}
\mmp\{\max_{0\leq k \leq n}{(S_k + \eta_{k + 1})} > (\mu + \varepsilon)n\} \geq \mmp\{\max_{0\leq k \leq n}{(S_k + \eta_{k + 1})} > (\mu + \varepsilon)n, N+1 \leq a_n\}\\ \geq \mmp\{\max_{a_n\leq k \leq n}{((\mu-\varepsilon)k+\eta_{k+1})}>(\mu + \varepsilon)n, N+1 \leq a_n\} = \mmp\{\max_{a_n\leq k \leq n}\,{((\mu - \varepsilon)k + \eta_{k + 1})} >(\mu + \varepsilon)n\}\\ - \mmp\{\max_{a_n\leq k \leq n}\,{((\mu - \varepsilon)k + \eta_{k + 1})} > (\mu + \varepsilon)n, N+1 > a_n\}.
\end{multline*}
Condition (W2) ensures $\lim_{n\to\infty}\mmp\{\max_{0\leq k \leq n}{(S_k + \eta_{k + 1})}>(\mu+\varepsilon)n\}=0$. This together with $\lim_{n\to\infty}\mmp\{N+1>a_n\} = 0$ proves
\begin{equation*}
\mmp\{\max_{a_n\leq k\leq n}((\mu - \varepsilon)k + \eta_{k + 1})\leq (\mu + \varepsilon)n\} = \prod_{k = a_n}^n F((\mu + \varepsilon)n-(\mu - \varepsilon)k)~\to~1,\quad n\to\infty,
\end{equation*}
where $F(x):=\mmp\{\eta\leq x\}$ for $x\in\mr$. Equivalently, $$\lim_{n\to\infty} \sum_{k=a_n}^{n}(-\log F((\mu + \varepsilon)n - (\mu - \varepsilon)k))=0.$$ By monotonicity,
\begin{multline*}
\sum_{k=a_n}^{n}(-\log F(n(\mu + \varepsilon) - (\mu - \varepsilon)k))\geq (n-a_n+1)(-\log F((\mu + \varepsilon)n - (\mu - \varepsilon)a_n)\\\geq (n-a_n)(-\log F((\mu+\varepsilon)n))
\end{multline*}
and thereupon $\lim_{n\to\infty}n(1-F((\mu+\varepsilon)n))=0$. Appealing to monotonicity once again we arrive at $({\rm W}3)$.

\noindent $({\rm W}3) \Rightarrow ({\rm W}2)$. Using $\max_{0 \leq k \leq n}\,(S_k + \eta_{k + 1})\geq S_n + \eta_{n + 1}$ a.s., we infer, for any $\varepsilon\in (0,\mu)$,
\begin{multline*}
\mmp\{\max_{0 \leq k \leq n}\,(S_k +\eta_{k+1})< (\mu -\varepsilon)n\}\leq \mmp\{S_n+\eta_{n+1}< (\mu-\varepsilon)n\} \\ \leq \mmp\{S_n< (\mu - \varepsilon/2)n\}+\mmp\{\eta_{n+1}< -\varepsilon n/2\}.
\end{multline*}
While the first term converges to $0$ as $n\to\infty$ by the weak law of large numbers for $(S_n)$, the second term does so trivially. Thus,
\begin{equation*}
\lim_{n\to\infty}\mmp\{\max_{0 \leq k \leq n}\,(S_k + \eta_{k + 1})<(\mu - \varepsilon)n\}= 0.
\end{equation*}
Observe that this limit relation holds true irrespective of $({\rm W}3)$.

Left with proving that, for any $\varepsilon>0$,
\begin{equation}\label{eq:prob_lower_bound}
\lim_{n\to\infty}\mmp\{\max_{0\leq k\leq n}\,(S_k+\eta_{k+1})>(\mu+\varepsilon)n\}= 0,
\end{equation}
we first note that (W3) is equivalent to $\lim_{t\to\infty}t|\log\mmp\{\eta\leq t\}|=0$, and that the latter ensures $n^{-1}\max_{1\leq k\leq n}\,\eta_k^+ \overset{\mmp}{\to} 0$ as $n\to\infty$, where $\overset{\mmp}{\to}$ denotes convergence in probability. Indeed, for any $\delta>0$, $$\mmp\{\max_{1\leq k\leq n}\,\eta_k^+>\delta n\}=1-\exp(-n|\log \mmp\{\eta\leq \delta n\}|)~\to~ 0,\quad n\to\infty.$$
It is known (see, for instance, Theorem 12.1 on p.~75 in \cite{Gut:2009}) that $\lim_{n\to\infty} n^{-1}\max_{0\leq k\leq n}\,S_k=\mu$ a.s.\ and thereupon $\lim_{n\to\infty}n^{-1}(\max_{0\leq k\leq n}\,S_k+\max_{1\leq k\leq n}\,\eta_k^+)=
\mu$ in probability. With this at hand, relation \eqref{eq:prob_lower_bound} follows from $\max_{0\leq k\leq n}\,(S_k+\eta_{k+1})\leq \max_{0\leq k\leq n}\,S_k+\max_{1\leq k\leq n + 1}\,\eta_k^+$ a.s. The proof of the implication $({\rm W}3) \Rightarrow ({\rm W}2)$ is complete.

Although the equivalence $({\rm W}1)\Leftrightarrow ({\rm W}2)$ is trivial, we prove for completeness one implication.

\noindent $({\rm W}2)\Rightarrow ({\rm W}1)$. Fix any $\varepsilon>0$, put $\delta:=\mu^{-1}(\mu+\varepsilon)^{-1}\varepsilon$ and observe that as $\varepsilon$ runs over $(0,\infty)$, $\delta$ sweeps out the interval $(0,\mu^{-1})$. For $n\in\mn_0:=\{0,1,2,\ldots\}$, put $s_n:=(\mu+\varepsilon)n$. According to $({\rm W}2)$,
\begin{multline*}
\mmp\{\tau(s_n)\leq (\mu^{-1}-\delta)s_{n+1}\}=\mmp\{\tau(s_n)\leq (\mu+\varepsilon)^{-1}s_{n+1}\}\\=\mmp\{\max_{
1\leq k\leq n+1}\, 
T_k>(\mu+\varepsilon)n\}~\to~0,\quad n\to\infty.
\end{multline*}
Given $t\geq 0$ there exists $n\in\mn_0$ such that $t\in [s_n, s_{n+1})$. Using this we infer $$\mmp\{\tau(t)\leq (\mu^{-1}-\delta)t\}\leq \mmp\{\tau(s_n)\leq (\mu+\varepsilon)^{-1}s_{n+1}\},$$ thereby proving that, for any $\delta\in (0,\mu^{-1})$, $\lim_{t\to\infty} \mmp\{\tau(t)\leq (\mu^{-1}-\delta)t\}=0$.  The relation $\lim_{t\to\infty} \mmp\{\tau(t)>(\mu^{-1}+\delta)t\}=0$, for any $\delta>0$, can be proved analogously. We omit further details.

\subsection{Proof of Theorem \ref{v_t}}\label{proof_v_t}

As we have already mentioned, the implication $({\rm S}3)\Rightarrow ({\rm S}1)$ was proven in Lemma 4.2 of \cite{Basrak etal:2022}.

\noindent $({\rm S}1)\Rightarrow ({\rm S}2)$. We shall use an alternative representation
\begin{equation}\label{eq:alternative}
\tau(t)=\inf\{n\in\mn: \max_{1\leq j\leq n}\,T_j>t\},\quad t\in\mr.
\end{equation}
By Lemma \ref{lem:div}, the assumption $\mu\in (0,\infty)$ ensures that $\lim_{n\to\infty}\max_{1\leq j\leq n}\,T_j=+\infty$ a.s. Since $\tau(-1/n+\max_{1\leq j\leq n}\,T_j)\leq n$ a.s.\ for each $n\in\mn$
we obtain
\begin{equation*}
\limsup_{n\to\infty}\frac{\max_{1\leq j\leq n}\,T_j}{n}\leq \limsup_{n\to\infty}\frac{\max_{1\leq j\leq n}\,T_j}{\tau(-1/n+\max_{1\leq j\leq n}\,T_j)}=\mu\quad\text{a.s.}
\end{equation*}
having utilized (S1). On the other hand, $\tau(\max_{1\leq j\leq n}\,T_j)=n+1+I_n \geq n+1$ a.s.\ for each $n\in\mn$, where $I_n:=\#\{k\in\mn: \max_{1\leq j\leq n+k}\,T_j=\max_{1\leq j\leq n}\,T_j \}$, whence
\begin{equation*}
\liminf_{n\to\infty}\frac{\max_{1\leq j\leq n}\,T_j}{n+1}\geq \liminf_{n\to\infty}  
\frac{\max_{1\leq j\leq n}\,T_j}{\tau(\max_{1\leq j\leq n}\,T_j)}=\mu\quad\text{a.s.}
\end{equation*}
by another appeal to (S1).

\noindent $({\rm S}2)\Rightarrow ({\rm S}3)$. Start with
\begin{equation*}
\limsup_{n\to\infty}\frac{T_n}{n} \leq \limsup_{n\to\infty}\frac{\max_{1\leq j\leq n}\,T_j}{n}=\mu \quad\text{a.s.}
\end{equation*}
This in combination with the strong law of large numbers for standard random walks entails $\limsup_{n\to\infty}n^{-1}\eta_n \leq 0$ a.s. With this at hand, an application of the converse part of the Borel-Cantelli lemma yields $\sum_{n\geq 1} \mmp\{\eta_n>\varepsilon n\} < \infty$ for all $\varepsilon>0$. This ensures $\me[\eta^+]<\infty$, as desired.

\subsection{Proof of Theorem \ref{N_t}}\label{proof_N_t}

By Proposition \ref{assert:nrho}, the a.s.\ finiteness of $N(t)$ or $\rho(t)$ for at least one deterministic $t$
is equivalent to $\me[\eta^-]<\infty$. According to \eqref{eq:ineq}, for each $t\in\mr$, $N(t)\leq \rho(t)$. Hence, it suffices to prove that
\begin{equation}\label{eq:inter}
\limsup_{t\to\infty} \frac{\rho(t)}{t}\leq\frac{1}{\mu}\leq \liminf_{t\to\infty} \frac{N(t)}{t}\quad\text{a.s.}
\end{equation}
To this end, put $\hat T_n:=S_{n-1}-\eta^-_n$ for $n\in\mn$ and $\hat \rho(t):=\sup\{n\geq 1: \hat T_n \leq t\}$ for $t\in\mr$. Since $\me [(-\eta^-)^-]=\me[\eta^-]<\infty$, we infer $\hat \rho(t)<\infty$ a.s.\ for all $t\in\mr$.  Further, $\hat T_n \leq T_n$ a.s.\ for $n\in\mn$, whence $\rho(t)\leq \hat\rho(t)$ a.s.\ for $t\in\mr$. Also, for each $t\in\mr$, $\hat T_{\hat \rho(t)} \leq t$ a.s.\ on $\{
\hat \rho(t)<\infty\}$. By the strong law of large numbers for standard random walks $\lim_{n\to\infty}(S_n/n)=\mu$ a.s. In view of $\me[\eta^-]<\infty$ and the Borel-Cantelli lemma, $\lim_{n\to\infty}(\eta_n^-/n)=0$ a.s. As a consequence, $\lim_{n\to\infty}(\hat T_n/n)=\mu$ a.s.\ and thereupon $\lim_{t\to\infty}(\hat T_{\hat \rho(t)}/\hat \rho(t))=\mu$ a.s.\ because $\lim_{t\to\infty}\hat \rho(t)=+\infty$ a.s. Indeed, by monotonicity, the a.s.\ limit $\lim_{t\to\infty}\hat \rho(t)$ exists, finite or infinite. In view of $\hat T_{\hat \rho(t)+1}>t$ a.s.\ on $\{\hat\rho(t)<\infty\}$, the limit cannot be finite. Combining pieces together we conclude that
\begin{equation*}
\limsup_{t\to\infty}\frac{\rho(t)}{t}\leq \limsup_{t\to\infty}\frac{\hat \rho(t)}{t} \leq \lim_{t\to\infty} \frac{\hat \rho(t)}{\hat T_{\hat \rho(t)}}=\frac{1}{\mu}\quad\text{a.s.}
\end{equation*}
This proves the left-hand inequality in \eqref{eq:inter}.

Recall that $\rho^\ast(t)=\sup\{k\geq 0: S_k \leq t\}$ for $t\in\mr$ and write, for any fixed $y>0$,
\begin{multline*}
N(t)=\sum_{k\geq 1}\1_{\{T_k\leq t\}}\geq \sum_{k=1}^{\rho^\ast(t)+1}\1_{\{T_k\leq t\}}\geq
	\sum_{k=1}^{\rho^\ast(t)+1}\1_{\{S_{k-1}\leq t-y\}}-\sum_{k=1}^{\rho^\ast(t)+1}\1_{\{\eta_k>y\}}\\=\sum_{k\geq 0}\1_{\{S_k\leq t-y\}}-\sum_{k=1}^{\rho^\ast(t)+1}\1_{\{\eta_k>y\}},\quad t\in\mr\quad\text{a.s.}
\end{multline*}
Here, we have used the inclusion $\{S_{k-1}\leq t-y\}\subseteq \{S_{k-1}+\eta_k\leq t\}\cup \{\eta_k>y\}$ for $k\in\mn$.
By Proposition \ref{prop:slln},
\begin{equation}\label{eq:inter1}
\lim_{t\to\infty}\frac{\sum_{k\geq 0}\1_{\{S_k\leq t-y\}}}{t}=\frac{1}{\mu}\quad\text{a.s.}	
\end{equation}
By the strong law of large numbers for standard random walks, $\lim_{n\to\infty}n^{-1}\sum_{k=1}^n\1_{\{\eta_k>y\}}=\mmp\{\eta>y\}$ a.s. Since $\lim_{t\to\infty}\rho^\ast(t)=+\infty$ a.s., we infer $\lim_{t\to\infty}(\rho^\ast(t)+1)^{-1}\sum_{k=1}^{\rho^\ast(t)+1}\1_{\{\eta_k>y\}}=\mmp\{\eta>y\}$ a.s. This in combination with $\lim_{t\to\infty}t^{-1}\rho^\ast(t)=\mu^{-1}$ a.s. (see Proposition \ref{prop:slln}) proves
\begin{equation}\label{eq:inter2}
\lim_{t\to\infty}\frac{\sum_{k=1}^{\rho^\ast(t)+1}\1_{\{\eta_k>y\}}}{t}=\lim_{t\to\infty}\frac{\sum_{k=1}^{\rho^\ast(t)+1}\1_{\{\eta_k>y\}}}{\rho^\ast(t)+1}\frac{\rho^\ast(t)+1}{t}= \frac{\mmp\{\eta>y\}}{\mu}\quad\text{a.s.}
\end{equation}
Invoking \eqref{eq:inter1} and \eqref{eq:inter2} yields
\begin{equation*}
\liminf_{t\to\infty} \frac{N(t)}{t}\geq \lim_{t\to\infty} \frac{\sum_{k\geq 0}\1_{\{S_k\leq t-y\}}}{t}-\lim_{t\to\infty} \frac{\sum_{k=1}^{\rho^\ast(t)+1}\1_{\{\eta_k>y\}}}{t}=\frac{\mmp\{\eta\leq y\}}{\mu}\quad\text{a.s.}
\end{equation*}
Letting now $y\to\infty$ we arrive at the right-hand inequality in \eqref{eq:inter}.

\section{Proofs of the functional limit theorems}

\subsection{Proof of Theorem \ref{thm1}}\label{proof_thm1}
By Theorem 1.3.15 in \cite{Iksanov:2016},
\begin{equation*}
(t^{-1}\max_{1\leq k \leq \lfloor ut\rfloor+1}\,T_k)_{u\geq 0}~ \Longrightarrow~ \big(\sup_{k:\, t_k^{(c,\,1)}\leq u}\,(\mu t_k^{(c,\,1)}+j_k^{(c,\,1)})\big)_{u\geq 0}=:(R(u))_{u\geq 0},\quad t\to\infty
\end{equation*}
in the $J_1$-topology on $D$. Since the first passage time functional is continuous in the $M_1$-topology (see, for instance, Lemma on p.~419 in \cite{Whitt:1971}), we conclude that
\begin{multline*}
(t^{-1}(\tau(ut)-1))_{u\geq 0}=\big(\inf\{z\geq 0: \max_{1\leq k\leq \lfloor zt \rfloor+1} T_k >ut\}\big)_{u\geq 0}\\~\Longrightarrow~ \big(\inf\big\{z \geq 0: \sup_{t_k^{(c,\,1)}\leq z}(\mu t_k^{(c,\,1)} + j_k^{(c,\,1)}) >u\big\}\big)_{u\geq 0}=(X(u))_{u\geq 0},\quad t\to\infty
\end{multline*}
in the $M_1$-topology on $D$.

For $y,u\geq 0$, put $F(y, u):=\mmp\{R(y)\leq u\}$. An explicit formula for $F$ is given in Remark 1.3.16 of \cite{Iksanov:2016}. Since $X$ is a nonnegative process, we conclude that $\mmp\{X(u)\leq y\}=0$ for $y<0$. For $y\geq 0$, formula \eqref{eq:marginal} is a consequence of $\mmp\{X(u)\leq y\}=1-F(y, u)$.

\subsection{Proof of Theorem \ref{thm2}}\label{proof_thm2}

Let $a$ be any positive function satisfying $\lim_{x\to\infty}x\mmp\{\eta>a(x)\}=1$. It is a standard fact (see, for instance, Lemma 6.1.3 in \cite{Iksanov:2016}) that $a$ is regularly varying at $\infty$ of index $1/\alpha$. In particular,
\begin{equation}\label{eq:in1}
\lim_{x\to\infty}\frac{a(x)}{x}=+\infty.
\end{equation}
By Theorem 1.8.3 in \cite{Bingham+Goldie+Teugels:1989}, there exists a continuous and strictly increasing function $b$ satisfying $b(x)\sim a(x)$ as $x\to\infty$. As a consequence, $\lim_{x\to\infty}x\mmp\{\eta>b(x)\}=1$. Thus, without loss of generality  we can and do assume that $a$ is continuous and strictly increasing.

We claim that
\begin{equation}\label{eq:conv}
\Big(\frac{\max_{1\leq k\leq \lfloor ut\rfloor+1}T_k}{a(t)}\Big)_{u\geq 0}~ \Longrightarrow~ \big(\sup_{k:\,t_k^{(1,\, \alpha)}\leq u}\,j_k^{(1,\, \alpha)}\big)_{u\geq 0},\quad t\to\infty
\end{equation}
in the $J_1$-topology on $D$. This follows along the lines of the proof of Proposition 1.3.13 (ii) in \cite{Iksanov:2016} provided that we can show that $$\Big(\frac{S_{\lfloor ut\rfloor}}{a(t)}\Big)_{u\geq 0}~\Longrightarrow~(\Theta(u))_{u\geq 0},\quad t\to\infty$$ in the $J_1$-topology on $D$, where $\Theta(u):=0$ for $u\geq 0$. The latter is an immediate consequence of \eqref{eq:in1} and the functional law of large numbers $$\Big(\frac{S_{\lfloor ut\rfloor}}{t}\Big)_{u\geq 0}~\Rightarrow~(I(u))_{u\geq 0},\quad t\to\infty$$ in the $J_1$-topology on $D$, where $I(u):=\mu u$ for $u\geq 0$.

With \eqref{eq:conv} at hand, invoking once again continuity of the first passage time functional in the $M_1$-topology we infer
\begin{multline}
\Big(\frac{\tau(ua(t))-1}{t}\Big)_{u\geq 0}=\big(\inf\{z\geq 0: \max_{1\leq k\leq \lfloor zt\rfloor+1}\,T_k > ua(t)\}\big)_{u\geq 0}\\~ \Longrightarrow~ \big(\inf\{z\geq 0:\sup_{k:\,t_k^{(1,\,\alpha)}\leq z}\,j_k^{(1,\,\alpha)}>u\}\big)_{u\geq 0},\quad t\to\infty \label{eq:in2}
\end{multline}
in the $M_1$-topology on $D$. Let $a^{-1}$ be the inverse function of $a$. Observe that $a$ is asymptotically generalized inverse of $x\mapsto (\mmp\{
\eta>x\})^{-1}$. Hence, $a^{-1}(x)\sim (\mmp\{
\eta>x\})^{-1}$ as $x\to\infty$. Substituting now $a^{-1}(t)$ in place of $t$ on the left-hand side of \eqref{eq:in2} we arrive at \eqref{eq:inverseextremal}.

Finally, we point out the marginal distributions of the limit process $Y$: for $x\geq 0$ and $u>0$,
\begin{multline*}
\mmp\{Y(u)>x\}=\mmp\big\{\inf\{z\geq 0:\,\sup_{k:\,t_k^{(1,\,\alpha)}\leq z}\,j_k^{(1,\,\alpha)}>u\big\}> x\}=\mmp\big\{\sup_{k:\,t_k^{(1,\, \alpha)}\leq x}\,j_k^{(1,\,\alpha)}\leq u\big\}\\=\mmp\big\{N^{(1,\,\alpha)}((t,y): t\leq x, y>u)=0\big\}=\exp\big(-\me N^{(1,\,\alpha)}((t,y): t\leq x, y>u)\big)=\eee^{-u^{-\alpha}x}.
\end{multline*}

\subsection{Proof of Theorem \ref{thm4}}\label{proof_thm4}

Under the assumptions $\mu\in(-\infty, +\infty)$ and $\sigma^2\in (0,\infty)$, a specialization of Donsker's theorem to
finite-dimensional distributions yields
\begin{equation}\label{donsker}
\Big(\frac{S_{\lfloor ut\rfloor} - \mu ut}{\sigma t^{1/2}}\Big)_{u\geq 0}~\fdc~ \big(B(u)\big)_{u\geq0}, \quad t\to\infty.
\end{equation}

Assume that $\mu\in(0, \infty)$, $\sigma^2\in (0,\infty)$ and $\me[\eta^+]<\infty$. We start by proving that
\begin{equation}\label{eq:inter4}
\Big(\frac{\max_{1\leq j\leq \lfloor ut\rfloor+1}\,T_j-\mu ut}{\sigma t^{1/2}}\Big)_{u\geq0}~\fdc~ \big(B(u)\big)_{u\geq 0}, \quad t\to\infty.
\end{equation}
For $u,t\geq 0$,
\begin{equation*}
\max_{
1\leq k \leq \lfloor ut\rfloor+1}\,T_k = S_{\lfloor ut\rfloor} + \max(\eta_{\lfloor ut\rfloor + 1}, \eta_{\lfloor ut\rfloor} - \xi_{\lfloor ut\rfloor}, \ldots, \eta_1 - S_{\lfloor ut\rfloor})\quad\text{a.s.}
\end{equation*}
In view of \eqref{donsker} it is enough to prove that
\begin{equation*}
\Big(\frac{\max(\eta_{\lfloor ut\rfloor + 1}, \eta_{\lfloor ut\rfloor} - \xi_{\lfloor ut\rfloor}, \ldots, \eta_1 - S_{\lfloor ut\rfloor})}{t^{1/2}}\Big)_{u\geq0}~\fdc~ \big(\Theta(u)\big)_{u\geq0}, \quad t\to\infty,
\end{equation*}
where $\Theta(u)=0$ for $u \geq 0$. According to the Cram\'{e}r-Wold device, this task is equivalent to showing that, for all $k\in\mn$, any $\lambda_1, \ldots, \lambda_k \in \mathbb{R}$ and any nonnegative $u_1, \ldots, u_k$,
\begin{equation*}
\lim_{t\to\infty} \sum_{i=1}^k \lambda_i \frac{\max(\eta_{\lfloor u_i t\rfloor + 1}, \eta_{\lfloor u_it\rfloor} - \xi_{\lfloor u_it\rfloor}, \ldots, \eta_1 - S_{\lfloor u_it\rfloor})}{t^{1/2}}=0\quad\text{in probability}.
\end{equation*}
Plainly, this is a consequence of
\begin{equation}\label{eq:inter3}
\lim_{t\to\infty}\frac{\max(\eta_{\lfloor u_0 t\rfloor + 1}, \eta_{\lfloor u_0t\rfloor} - \xi_{\lfloor u_0t\rfloor}, \ldots, \eta_1 - S_{\lfloor u_0t\rfloor})}{t^{1/2}}=0\quad\text{in probability},
\end{equation}
where $u_0\geq 0$ is fixed. If $u_0=0$, then \eqref{eq:inter3} trivially holds. Thus, we assume in what follows that $u_0>0$.

\noindent {\sc Proof of \eqref{eq:inter3}}. Observe that, for all $\varepsilon>0$,
\begin{multline*}
\mmp\big\{\max(\eta_{\lfloor u_0t\rfloor + 1}, \eta_{\lfloor u_0t\rfloor} - \xi_{\lfloor u_0t\rfloor}, \ldots, \eta_1 - S_{\lfloor u_0t\rfloor})< -\varepsilon t^{1/2}\big\} \leq \mmp\{\eta_{\lfloor u_0t\rfloor + 1}<-\varepsilon t^{1/2}\}\\=\mmp\{\eta<-\varepsilon t^{1/2}\}~\to~ 0,\quad t\to\infty.
\end{multline*}
To proceed, let $u_0t\geq 1$ and $\eta_0$ be a copy of $\eta$ which is independent of $(\xi_1,\eta_1)$, $(\xi_2, \eta_2),\ldots$. The maximum in \eqref{eq:inter3} has the same distribution as $$\max(\eta_0, \eta_1- \xi_1, \ldots, \eta_{\lfloor u_0t\rfloor} - S_{\lfloor u_0t\rfloor})=\max (\eta_0, \max_{1\leq k\leq \lfloor u_0 t\rfloor }\,\hat T_k),$$ where $(\hat T_n)_{n\geq 1}$ is a globally perturbed random walk generated by $(-\xi, \eta-\xi)$. Since $\me [-\xi]\in (-\infty, 0)$ and $\me [(\eta-\xi)^+]\leq \me [\eta^+]+\me [\xi^-]<\infty$ we infer $\lim_{n\to\infty}\hat T_n=-\infty$ a.s.\ by (a mirror version of) Proposition \ref{assert:nrho}. This ensures that the variable $\max (\eta_0, \max_{k\geq 1}\,\hat T_k)$ is a.s.\ finite, whence, for all $\varepsilon>0$,
\begin{equation*}
\mmp\big\{\max(\eta_{\lfloor u_0t\rfloor + 1}, \eta_{\lfloor u_0t\rfloor} - \xi_{\lfloor u_0t\rfloor}, \ldots, \eta_1 - S_{\lfloor u_0t\rfloor}) > \varepsilon t^{1/2}\big\} \leq \mmp\big\{\max (\eta_0, \max_{k\geq 1}\,\hat T_k)>\varepsilon t^{1/2}\big\}~ \to~ 0
\end{equation*}
as $t\to\infty$, thereby completing the proof of \eqref{eq:inter3} and \eqref{eq:inter4}.

We are ready to prove \eqref{eq:tau}. To this end, we have to show that for all $k\in\mn$, any nonnegative $u_1, \ldots, u_k$ and any real $x_1, \ldots, x_k$,
\begin{equation}\label{eq:inter5}
\lim_{t\to\infty}\mmp\Big\{\frac{\tau(u_1t) - \mu^{-1}u_1 t}{(\sigma^2\mu^{-3}t)^{1/2}} > x_1, \ldots, \frac{\tau(u_k t) - \mu^{-1}u_k t}{(\sigma^2\mu^{-3}t)^{1/2}} > x_k\Big\} = \mmp\{B(u_1) > x_1, \ldots, B(u_k) > x_k\}.
\end{equation}
If $u_i=0$, then $(\tau(u_it) - \mu^{-1}u_i t)(\sigma^2\mu^{-3}t)^{-1/2}=\tau(0)(\sigma^2 \mu^{-3}t)^{-1/2}$ a.s., and, as $t\to\infty$, this converges to $0$. Since $B(u_i)=B(0)=0$ a.s., we assume in what follows that $u_1,\ldots, u_k$ are positive. As a preparation, for fixed real $z_1,z_2$, $z_2\neq 0$, put $\ell(t,z_1,z_2):= \mu^{-1}t+\frac{z_1 (\sigma^2\mu^{-3}t)^{1/2}-1}{z_2}$ and observe that
\begin{equation}\label{supp_k_limit}
\frac{t-\mu \ell(t, z_1,z_2)}{(\sigma^2\ell(t, z_1,z_2))^{1/2}} = -\frac{z_1}{z_2}\Big(\frac{t}{t + (z_1(\sigma^2\mu^{-1}t)^{1/2}-\mu)z_2^{-1}
}\Big)^{1/2}+o(1)~ \to~ -\frac{z_1}{z_2}, \quad t\to\infty.
\end{equation}
For $t$ so large that $u_i \ell(t, x_i, u_i)\geq 1$ for $1\leq i\leq k$, the probability on the left-hand side of \eqref{eq:inter5} is equal to
\begin{multline*}
\mmp\{\tau(u_1 t) > u_1\ell(t, x_1, u_1)+1, \ldots, \tau(u_k t) > u_k\ell(t, x_k, u_k)+1\}\\=
\mmp\Big\{\bigcap_{i=1}^k \Big\{\frac{\max_{1\leq j \leq \lfloor u_i \ell(t, x_i, u_i)\rfloor+1}\,T_j - \mu u_i \ell(t, x_i, u_i)}{\sigma (\ell(t, x_i, u_i))^{1/2}} \leq \frac{u_i(t - \mu \ell(t, x_i, u_i))}{\sigma(\ell(t, x_i, u_i))^{1/2}}\Big\}\Big\}.
\end{multline*}
We have used representation \eqref{eq:alternative} for the equality. In view of \eqref{eq:inter4} and \eqref{supp_k_limit}, as $t\to\infty$, the last probability converges to $\mmp\{B(u_1) \leq -x_1, \ldots, B(u_k) \leq -x_k\}=\mmp\{B(u_1) > x_1, \ldots, B(u_k) > x_k\}$. The equality is justified by the fact that $B(u_i)$ has the same continuous distribution as $-B(u_i)$. The proof of \eqref{eq:tau} is complete.

Assume now that $\mu\in(0, \infty)$, $\sigma^2\in (0,\infty)$ and $\me[\eta^-]<\infty$. Our proof of \eqref{eq:rho} is similar to that of \eqref{eq:tau}. Hence, we only give a sketch. We claim that
\begin{equation}\label{inf_flt}
\Big(\frac{\inf_{k\geq \lfloor ut\rfloor+1}\,T_k - \mu ut}{\sigma t^{1/2}}\Big)_{u\geq 0}~\fdc~\big(B(u)\big)_{u\geq0}, \quad t\to\infty.
\end{equation}
Observe that, for $n\geq 0$, $\inf_{k \geq n+1}\,T_k= S_n+\inf\{\eta_{n+1}, \eta_{n+2}+\xi_{n+1}, \eta_{n+3}+\xi_{n+1}+\xi_{n+2},\ldots\}$, and the latter infimum has the same distribution as $\inf_{k \geq 1}\,T_k$. According to Proposition \ref{assert:nrho}, the a.s. finiteness of the infimum is secured by the assumptions $\mu\in(0, \infty)$ and $\me[\eta^-]<\infty$. Now \eqref{inf_flt} follows from the latter equality and \eqref{donsker}.

To complete the proof of \eqref{eq:rho} we have to check that for all $k\in\mn$, any nonnegative $u_1, \ldots, u_k$ and any real $x_1, \ldots, x_k$,
\begin{equation}\label{eq:rho1}
\lim_{t\to\infty}\mmp\Big\{\frac{\rho(u_1t) - \mu^{-1}u_1 t}{(\sigma^2\mu^{-3}t)^{1/2}} > x_1, \ldots, \frac{\rho(u_k t) - \mu^{-1}u_k t}{(\sigma^2\mu^{-3}t)^{1/2}} >
x_k\Big\} = \mmp\{B(u_1) > x_1, \ldots, B(u_k) > x_k\}.
\end{equation}
For $t$ so large that $\ell(t, x_i, u_i)\geq 0$ for $1\leq i\leq k$, the probability on the left-hand side of \eqref{eq:rho1} is equal to
\begin{multline*}
\mmp\{\rho(u_1 t) > u_1\ell(t, x_1, u_1), \ldots, \rho(u_k t) > u_k\ell(t, x_k, u_k)\}\\=\mmp\Big\{\bigcap_{i=1}^k \Big\{\frac{\inf_{j \geq \lfloor u_i \ell(t, x_i, u_i)\rfloor+1}\,T_j - \mu u_i \ell(t, x_i, u_i)}{\sigma(\ell(t, x_i, u_i))^{1/2}} \leq \frac{u_i(t - \mu \ell(t, x_i, u_i))}{\sigma(\ell(t, x_i, u_i))^{1/2}}\Big\}\Big\}.
\end{multline*}
Arguing along the lines of the proof of \eqref{eq:tau}, but using \eqref{inf_flt} in place of \eqref{eq:inter4} we arrive at \eqref{eq:rho1}. The proof of \eqref{eq:rho} is complete.

Finally, we assume that $\mu\in(0, \infty)$, $\sigma^2\in (0,\infty)$ and $\me[\eta]\in (-\infty,+\infty)$. Formula \eqref{eq:mix} follows from the already proved relations \eqref{eq:tau} and \eqref{eq:rho} in combination with the following simple observation. If for each $t$ large enough and each $u\geq 0$, $X_t(u)\leq Y_t(u)\leq Z_t(u)$, $(X_t(u))_{u\geq 0}\fdc (X(u))_{u\geq 0}$ and $(Z_t(u))_{u\geq 0}\fdc (X(u))_{u\geq 0}$ as $t\to\infty$, then $$((X_t(u))_{u\geq 0}, (Y_t(u))_{u\geq 0}, (Z_t(u))_{u\geq 0})~\fdc~((X(u))_{u\geq 0}, (X(u))_{u\geq 0}, (X(u))_{u\geq 0}),\quad t\to\infty.$$ We use this fact with $X_t(u)=\tau(ut)-1$, $Y_t(u)=N(ut)$ and $Z_t(u)=\rho(ut)$. The proof of Theorem \ref{thm4} is complete.

\subsection{Proof of Theorem \ref{thm5}}\label{proof_thm5}

Let $\mathbb{Z}$ denote the set of integers. Recall that a function $f: \mr\to [0,\infty)$ is called {\it directly Riemann integrable} (dRi) on $\mr$, if

\noindent (a) $\bar \sigma(h)<\infty$ for each $h>0$ and

\noindent (b) $\lim_{h\to 0+}(\bar \sigma(h)-\ubar \sigma (h))= 0$, where
$$\bar \sigma(h):=h \sum_{k\in\mathbb{Z}}\sup_{(k-1)h\leq y<kh}\,f(y)\quad\text{and}\quad \ubar\sigma(h):=h \sum_{k\in\mathbb{Z}}\inf_{(k-1)h\leq y<kh}\,f(y).$$ A function $f: \mr\to \mr$ is called dRi on $\mr$, if so are $f^+$ and $f^-$, where $f^+(t)=\max (f(t),0)$ and $f^-(t)=\max (-f(t), 0)$ for $t\in\mr$.

We need a couple of auxiliary results.
\begin{lemma}\label{lem:gen_shot_noice_process}
Assume that $\me [\xi]\in (0,\infty)$ and let $f:\mr\to\mr$ be a dRi function. Then the function $t\to\me\big[\sum_{k\geq0}f(t-S_k)\big]$ is bounded.
\end{lemma}
Let $d>0$. Recall that the distribution of a real-valued random
variable $\theta$ is called {\it $d$-arithmetic} if it is concentrated on the lattice $(nd)_{n\in\mathbb{Z}}$ and not concentrated on $(nd_1)_{n\in\mathbb{Z}}$ for some $d_1>d$. The distribution is called {\it nonarithmetic}, if it is not $d$-arithmetic for any $d>0$. If in addition to the assumptions of Lemma \ref{lem:gen_shot_noice_process}, the distribution of $\xi$ is nonarithmetic, then, by Theorem 4.2 in \cite{Athreya+McDonald+Ney:1978}, $$\lim_{t\to\infty}\me \Big[\sum_{k\geq 0}f(t-S_k)\Big]=\mu^{-1}\int_\mr f(y){\rm d}y.$$ If $\me [\xi]\in (0,\infty)$, the distribution of $\xi$ is $d$-arithmetic for some $d>0$, and the series $\sum_{j\in\mathbb{Z}}f(jd)$ converges, then, by Proposition 2.1 in \cite{Iksanov+Polotsky:2016}, $$\lim_{n\to\infty}\me \Big[\sum_{k\geq 0}f(nd-S_k)\Big]=d\mu^{-1}\sum_{j\in\mathbb{Z}}f(jd).$$ Thus, the statement of Lemma \ref{lem:gen_shot_noice_process} could have been derived from these two results. However, we prefer to give an economical proof which does not require distinguishing the two cases (nonarithmetic vs arithmetic). Neither does it use the reduction to a standard random walk formed by strictly ascending ladder heights, that was exploited in both \cite{Athreya+McDonald+Ney:1978} and \cite{Iksanov+Polotsky:2016}.
\begin{proof}[Proof of Lemma \ref{lem:gen_shot_noice_process}]
We can investigate the sums generated by $f^+$ and $f^-$ separately. Hence, we assume that $f$ is nonnegative.

The assumption $\me [\xi]\in (0,\infty)$ entails transience of $(S_k)_{k\geq 0}$, which particularly guarantees that, for each $s>0$, $\me\Big[\sum_{k\geq 0}\1_{\{S_k\in (-s,\, s)\}}\Big]<\infty$. For $t\in\mr$ and $s>0$, put $$\nu:=\inf\{k\geq 0: S_k\in (t,\, t+s]\}.$$ The stopping time $\nu$ may be infinite, in which case $\sum_{k\geq 0}\1_{\{S_k\in (t,\, t+s]\}}=0$. Further,
\begin{multline*} 
\sum_{k\geq 0}\1_{\{S_k\in (t,\, t+s]\}}=\1_{\{\nu<\infty\}}\sum_{k\geq 0}\1_{\{S_{\nu+k}\in (t,\, t+s]\}}=\1_{\{\nu<\infty\}}\sum_{k\geq 0}\1_{\{S_{\nu+k}-S_{\nu}\in (t-S_{\nu},\, t+s-S_{\nu}]\}}\\\leq \1_{\{\nu<\infty\}} \sum_{k\geq 0}\1_{\{S_{\nu+k}-S_{\nu}\in (-s,\, s)\}}\quad \text{a.s.}
\end{multline*}
On the event $\{\nu<\infty\}$, $(S_{\nu+k}-S_{\nu})_{k\geq 1}$ has the same distribution as $(S_k)_{k\geq 1}$. Hence, passing to the expectations we infer
\begin{equation}\label{eq:finiteness}
\me\Big[\sum_{k\geq 0}\1_{\{S_k\in (t,\, t+s]\}}\Big]\leq \mmp\{\nu<\infty\}\me\Big[\sum_{k\geq 0}\1_{\{S_k\in (-s,\, s)\}}\Big]\leq \me\Big[\sum_{k\geq 0}\1_{\{S_k\in (-s,\, s)\}}\Big].
\end{equation}

Finally, put $\hat f(t):= \sum_{n\in\mathbb{Z}}\big(\sup_{y\in [n-1,\,n)}\, f(y)\big) \1_{[n-1,\, n)}(t)$ for $t\in\mr$ and observe that $f(t)\leq \hat f(t)$ for $t\in\mr$. Then
\begin{multline*}
\me\Big[\sum_{k\geq 0} f(t - S_k)\Big]\leq \me\Big[\sum_{k\geq 0}\hat f(t-S_k)\Big]=\sum_{
n\in\mathbb{Z}}\big(\sup_{y\in [n-1,\,n)}\,f(y)\big)\me\Big[\sum_{k\geq 0}\1_{\{S_k\in (t-n,\,t-n+1] \}}\Big]\\\leq \bar \sigma(1)\me\Big[\sum_{k\geq 0}\1_{\{S_k\in (-1,\,1)\}}\Big]<\infty.
\end{multline*}
Here, the first equality is justified by Fubini's theorem, and the second inequality is a consequence of \eqref{eq:finiteness}. The proof of Lemma \ref{lem:gen_shot_noice_process} is complete.
\end{proof}
\begin{lemma}\label{lem:intsum}
Let $\rho\geq 1$ and $f:\mr\to [0,\infty)$ be a locally bounded function. Then
$$\me \Big[\Big(\sum_{k\geq 0}f(t-S_k)\1_{\{0<S_k\leq t\}}\Big)^\rho\Big]\leq \Big(\sum_{j=0}^{\lfloor t\rfloor}\sup_{y\in [j,\,j+1)}\,f(y)\Big)^\rho \me\Big[\Big(\sum_{k\geq 0}\1_{\{S_k\in (-1,\,1)\}}\Big)^\rho\Big].$$
\end{lemma}
\begin{proof}
We start with $$f(t)\leq \sum_{n=0}^{\lfloor t\rfloor}\sup_{y\in[n,\,n+1)}\,f(y)\1_{[n,\,n+1)}(t),\quad t\geq 0.$$ Using this inequality and convexity of $x\mapsto x^\rho$ we obtain
\begin{multline}
\Big(\sum_{k\geq 0}f(t-S_k)\1_{\{0<S_k\leq t\}}\Big)^\rho \leq \Big(\sum_{n=0}^{\lfloor t\rfloor}\big(\sup_{y\in [n,\,n+1)}\,f(y)\big)\sum_{k\geq 0}\1_{\{S_k\in (t-(n+1),\,t-n]\}}\Big)^\rho\\\leq \Big(\sum_{j=0}^{\lfloor t\rfloor}\sup_{y\in [j,\,j+1)}\,f(y)\Big)^\rho \sum_{n=0}^{\lfloor t\rfloor}\frac{\sup_{y\in [n,\,n+1)}\,f(y)}{\sum_{j=0}^{\lfloor t\rfloor}\sup_{y\in [j,\,j+1)}\,f(y)}\Big(\sum_{k\geq 0}\1_{\{S_k\in (t-(n+1),\,t-n]\}}\Big)^\rho.\label{eq:estim}
\end{multline}

Now we prove that
\begin{equation}\label{eq:mom2}
\me \Big[\Big(\sum_{k\geq 0}\1_{\{S_k\in (-1,\,1)\}}\Big)^\rho\Big]<\infty.
\end{equation}
This holds trivially if $(S_k)_{k\geq 0}$ is a $d$-arithmetic random walk with $d\geq 1$. Assume that $(S_k)_{k\geq 0}$ is either nonarithmetic or $d$-arithmetic with $d\in (0,1)$. Put $\nu_0:=\inf\{k\geq 1: S_k\in (-1,1)\}$ if $S_k\in (-1,\,1)$ for some $k\in\mn$ and $\nu_0:=+\infty$ otherwise. The random variable $\sum_{k\geq 0}\1_{\{S_k\in (-1,\,1)\}}$ is stochastically dominated by $Q$, a random variable with a geometric distribution with success probability $p:=\mmp\{\nu_0=
\infty|S_0\in (-1,1)\}$, that is, $\mmp\{Q=k\}=(1-p)^{k-1}p
$ for $k\in\mn$. This entails \eqref{eq:mom2}.

Now analogously to \eqref{eq:finiteness} we conclude that $$\me\Big[\Big(\sum_{k\geq 0}\1_{\{S_k\in (t-(n+1),\, t-n]\}}\Big)^\rho\Big]\leq \me\Big[\Big(\sum_{k\geq 0}\1_{\{S_k\in (-1,\, 1)\}}\Big)^\rho\Big]$$ and the claim of the lemma follows upon passing to the expectations in \eqref{eq:estim}.
\end{proof}

Let $(R_k)_{k\geq 0}$ be a not necessarily monotone sequence of nonnegative random variables. Put $$M(t):=\sum_{k\geq 0}\1_{\{R_k\leq t\}},\quad t\geq 0$$ and assume that, for each $t\geq 0$, $M(t)<\infty$ a.s. Now we state a very particular version of Theorem 1.1 and Remark 1.1 in \cite{Iksanov+Rashytov:2020}.
\begin{lemma}\label{lem:Iks}
Let $h\in D$ be a nondecreasing function satisfying $\lim_{t\to\infty}h(t)=a\in (0,\infty)$. Assume that, for some positive constants $b$ and $c$, $$\Big(\frac{M(ut)-b^{-1}ut}{c\,t^{1/2}}\Big)_{u\geq 0}~\Longrightarrow~ (B(u))_{u\geq 0},\quad t\to\infty$$ in the $J_1$-topology on $D$, where $(B(u))$ is a standard Brownian motion. Then $$\frac{\sum_{k\geq 0}h(ut-R_k)\1_{\{R_k\leq ut\}}-b^{-1}\int_0^{ut}h(y){\rm d}y}{ac\,t^{1/2}}~\Longrightarrow~(B(u))_{u\geq 0},\quad t\to\infty$$ in the $J_1$-topology on $D$.
\end{lemma}

\begin{proof}[Proof of Theorem \ref{thm5}]
Recall the notation $F(t)=\mmp\{\eta\leq t\}$ for $t\in\mr$. Here is a basic decomposition for what follows:
\begin{multline*}
N(t)-\mu^{-1}\int_0^t F(y){\rm d}y=\sum_{k\geq 0}\1_{\{S_k+\eta_{k+1}\leq t,\,S_k\leq 0\}}+\sum_{k\geq 0}\1_{\{S_k+\eta_{k+1}\leq t,\,S_k>t\}}\\+\sum_{k\geq 0}\big(\1_{\{S_k+\eta_{k+1}\leq t\}}-F(t-S_k)\big)\1_{\{0<S_k\leq t\}}+\Big(\sum_{k\geq 0}F(t-S_k)\1_{\{0<S_k\leq t\}}-\mu^{-1}\int_0^t F(y){\rm d}y\Big)\\=:\sum_{r=1}^4 I_r(t).
\end{multline*}
The third term is a `martingale term' (the terminal value of a martingale), and the fourth term is a centered `shot noise term', whose asymptotic behavior is driven by Lemma \ref{lem:Iks}. We shall show that the fourth term gives a principal contribution, whereas all the other terms vanish in the limit.

\noindent {\sc Analysis of $I_1$.} Plainly, for all $T>0$, $$0\leq t^{-1/2}\sup_{u\in [0,\,T]}\, I_1(ut)\leq t^{-1/2}N^\ast(0)~\to~0,\quad t\to\infty\quad\text{a.s.}$$ Here, we have used $N^\ast(0)<\infty$ a.s. which is secured by $\lim_{n\to\infty}S_n=+\infty$ a.s.

\noindent {\sc Analysis of $I_2$.}
We first show that
\begin{equation}\label{eq:momx}
\me [(I_2(t))^2]=O(1),\quad t\to\infty.
\end{equation}
To this end, write, for $t\in\mr$, $$I_2(t)=\sum_{k\geq 0}\big(\1_{\{S_k+\eta_{k+1}\leq t\}}-F(t-S_k)\big)\1_{\{S_k>t\}}+\sum_{k\geq 0}F(t-S_k)\1_{\{S_k>t\}}=:I_{21}(t)+I_{22}(t)$$ and observe that $\me [(I_2(t))^2]\leq 2(\me [(I_{21}(t))^2]+\me [(I_{22}(t))^2])$. Put $$m(t):=\me [I_{22}(t)]=\me \Big[\sum_{k\geq 0}F(t-S_k)\1_{\{S_k>t\}}\Big],\quad t\in\mr.$$ Since the function $t\mapsto F(t)$ is nondecreasing and $\int_{-\infty}^0 F(t){\rm d}t = \me [\eta^-] < \infty$, it is dRi on $(-\infty, 0]$. This follows, for instance, from Lemma 6.2.1 on p.~213 in \cite{Iksanov:2016} applied to $t\mapsto F(-t)$. As a consequence, the function $t\mapsto F(t)\1_{(-\infty, 0)}(t)$ is dRi on $\mr$. Hence, by Lemma \ref{lem:gen_shot_noice_process}, $m(t)\leq c$ for some constant $c>0$ and all $t\in\mr$.

We claim that
\begin{equation}\label{eq:formula}
\me [(I_{21}(t))^2]=\me\Big[\sum_{k\geq 0}F(t-S_k)(1-F(t-S_k))\1_{\{S_k>t\}}\Big],\quad t\in\mr.
\end{equation}
To prove this, write
\begin{multline*}
\me [(I_{21}(t))^2]=
\sum_{k\geq 0}\me \big[\big(\1_{\{S_k+\eta_{k+1}\leq t\}}-F(t-S_k)\big)^2\1_{\{S_k>t\}}\big]\\+2\,\me\Big[\sum_{0\leq i<j}\big(\1_{\{S_i+\eta_{i+1}\leq t\}}-F(t-S_i)\big)\1_{\{S_i>t\}}\big(\1_{\{S_j+\eta_{j+1}\leq t\}}-F(t-S_j)\big)\1_{\{S_j>t\}}\Big].
\end{multline*}
For $\ell\in\mn$, let $\mathcal{G}_\ell$ denote the $\sigma$-algebra generated by $(\xi_k,\eta_k)_{1\leq k\leq \ell}$ and $\mathcal{G}_0$ denote the trivial $\sigma$-algebra. Then, for $k\in\mn_0$, $$\me \big[\big(\1_{\{S_k+\eta_{k+1}\leq t\}}-F(t-S_k)\big)^2\1_{\{S_k>t\}}\big|\mathcal{G}_k\big]=F(t-S_k)(1-F(t-S_k))\1_{\{S_k>t\}}\quad\text{a.s.}$$ and, for $0\leq i<j$, $$\me \big[\big(\1_{\{S_i+\eta_{i+1}\leq t\}}-F(t-S_i)\big)\1_{\{S_i>t\}}\big(\1_{\{S_j+\eta_{j+1}\leq t\}}-F(t-S_j)\big)\1_{\{S_j>t\}}\big|\mathcal{G}_j\big]=0\quad\text{a.s.}$$ because the variable $\big(\1_{\{S_i+\eta_{i+1}\leq t\}}-F(t-S_i)\big)\1_{\{S_i>t\}}$ is $\mathcal{G}_j$-measurable and $$\me\big[ \big(\1_{\{S_j+\eta_{j+1}\leq t\}}-F(t-S_j)\big)\1_{\{S_j>t\}}\big|\mathcal{G}_j\big]=0\quad\text{a.s.}$$ Formula \eqref{eq:formula} follows from these facts and entails, for $t\in\mr$, $$\me [(I_{21}(t))^2] 
\leq \me\Big[\sum_{k\geq 0}F(t-S_k)\1_{\{S_k>t\}}\Big]\leq c.$$ Further, for $t\in\mr$, $$\me [(I_{22}(t))^2]=\me\Big[\sum_{k\geq 0}(F(t-S_k))^2\1_{\{S_k>t\}}\Big]+2\,\me\Big[\sum_{0\leq i<j}F(t-S_i)F(t-S_j)\1_{\{S_i>t\}}\1_{\{S_j>t\}}\Big].$$ The first term on the right-hand side is bounded from above by $c$. The second does not exceed
$$2\me\Big[\sum_{k\geq 0}F(t-S_k)m(t-S_k)\1_{\{S_k>t\}}\Big]\leq 2c m(t)\leq 2c^2,\quad t\in\mr.$$ Thus, \eqref{eq:momx} has been proved.

Pick $\delta\in (1,2)$. Given $t>0$ there exists $n\in\mn$ such that $t\in (n^\delta, (n+1)^\delta]$. Using monotonicity we obtain
\begin{multline}
0\leq t^{-1/2}I_2(t)\leq n^{-\delta/2} \sum_{k\geq 0}\1_{\{S_k+\eta_{k+1}\leq (n+1)^\delta,\,S_k>n^\delta\}}\\\leq n^{-\delta/2}I_2((n+1)^\delta)+n^{-\delta/2}\sum_{k\geq 0}\1_{\{S_k\in (n^\delta,\,(n+1)^\delta]\}}.\label{eq:inter100}
\end{multline}
In view of \eqref{eq:momx} and Markov's inequality, for all $\varepsilon>0$, $\mmp\{I_2((n+1)^\delta)>\varepsilon n^{\delta/2}\}=O(n^{-\delta})$ as $n\to\infty$. Hence, $\sum_{n\geq 1}\mmp\{I_2((n+1)^\delta)>\varepsilon n^{\delta/2}\}<\infty$ and, by the direct part of the Borel-Cantelli lemma, $\lim_{n\to\infty}n^{-\delta/2}I_2((n+1)^\delta)=0$ a.s.

Next, we intend to check that the second term on the right-hand side of \eqref{eq:inter100} converges to $0$ a.s., too. It is shown in the proof of Lemma \ref{lem:intsum} that, for all $r>1$, $\me [(\sum_{k\geq 0}\1_{\{S_k\in (-1,\,1)\}})^r]<\infty$. Choose minimal $r>1$ satisfying $r(2-\delta)>2$. For simplicity of presentation, assume that $n^\delta$ and $(n+1)^\delta$ are integer and put $a_n:=(n+1)^\delta-n^\delta$. Similarly to \eqref{eq:finiteness} we infer with the help of Minkowski's inequality that
\begin{multline*}
\me \Big[\Big(\sum_{k\geq 0}\1_{\{S_k\in (n^\delta,\,(n+1)^\delta]\}}\Big)^r\Big]=\me \Big[\Big(\sum_{j=n^\delta}^{(n+1)^\delta-1}\sum_{k\geq 0}\1_{\{S_k\in (j,\, j+1]\}}\Big)^r \Big]\\\leq (a_n)^r \me\Big[\Big(\sum_{k\geq 0}\1_{\{S_k\in (-1,\,1)\}}\Big)^r\Big].
\end{multline*}
Using $(a_n)^r n^{-r\delta/2}=O(n^{-r(2-\delta)/2})$, Markov's inequality and the direct part of the Borel-Cantelli lemma we conclude that
\begin{multline*}
\sum_{n\geq 1}\mmp\Big\{\sum_{k\geq 0}\1_{\{S_k\in (n^\delta,\,(n+1)^\delta]\}}>\varepsilon n^{\delta/2}\Big\}\leq \sum_{n\geq 1}\frac{(a_n)^r\me [(\sum_{k\geq 0}\1_{\{S_k\in(-1,\,1)\}})^r]}{\varepsilon^r n^{r\delta/2}}\\=\sum_{n\geq 1}O(n^{-r(2-\delta)/2})<\infty.
\end{multline*}
Thus, $\lim_{n\to\infty} n^{-\delta/2}\sum_{k\geq 0}\1_{\{S_k\in (n^\delta,\,(n+1)^\delta]\}}=0$ a.s. and thereupon, recalling \eqref{eq:inter100}, $$\lim_{t\to\infty} t^{-1/2} I_2(t)=0\quad\text{a.s.}$$ This entails, for all $T>0$, $\lim_{t\to\infty} t^{-1/2}\sup_{u\in [0,\,T]}I_2(ut)=0$ a.s., thereby showing that the contribution of $I_2$ is negligible.

\noindent {\sc Analysis of $I_3$.} It suffices to prove that, for all $T>0$, $t^{-1/2}\sup_{u\in [0,\,T]}\,|I_3(ut)|\overset{\mmp}{\to} 0$ as $t\to\infty$ (we note in passing that the latter limit relation does not necessarily hold, with the a.s.\ convergence replacing the convergence in probability, see Remark \ref{rem:as} for more details). Plainly, considering $T=1$ does the job.

Put $\kappa=\kappa(t):=\lfloor (3/4) \log_2 t\rfloor$ for $t\geq 1$. For $j\in\mn_0$ and $t>0$, put $$\mathcal{V}_j
(t):=\{v_{j,m}(t):=2^{-j}m t:0\leq m\leq 2^j, m\in\mn_0\}.$$ In what follows, we write $v_{j,\,m}$ for $v_{j,\,m}(t)$. Observe that $\mathcal{V}_j
(t)\subseteq \mathcal{V}_{j+1}(t)$. For any $u\in [0, t]$, put
$$
u_j:=\max\{v\in \mathcal{V}_j
(t): v\leq u\}=2^{-j}t\lfloor 2^j t^{-1} u\rfloor.
$$
Observe that either $u_{j-1}=u_j$ or $u_{j-1}=u_j-2^{-j}t$. Necessarily, $u_j=v_{j,\,m}$ for some $0\leq m\leq 2^j$, so that either $u_{j-1}=v_{j,\,m}$ or $u_{j-1}=v_{j,\,m-1}$. Write
\begin{multline*}
\sup_{u\in [0,\, t]}|I_3(u)|=\max_{0\leq j\leq 2^\kappa-1}\sup_{z\in [0,\,v_{\kappa,\,j+1}-v_{\kappa,\,j}]}|I_3(v_{\kappa,\,j})+(I_3(v_{\kappa,\,j}+z)-I_3(v_{\kappa,\,j}))|\\\leq \max_{0\leq j\leq 2^\kappa-1}\,|I_3(v_{\kappa,\,j})|+\max_{0\leq j\leq 2^\kappa-1}\sup_{z\in [0,\,v_{\kappa,\,j+1}-v_{\kappa,\,j}]}\,|I_3(v_{\kappa,\,j}+z)-I_3(v_{\kappa,\,j})|\quad\text{a.s.}
\end{multline*}

For $u\in \mathcal{V}_\kappa(t),$
\begin{equation*}
|I_3(u)| = \Big|\sum_{j=1}^\kappa (I_3(u_j)-I_3(u_{j-1}))+I_3(u_0)\Big|\leq \sum_{j=0}^\kappa \max_{1\leq m\leq 2^j}\,|I_3(v_{j,\,m})-I_3(v_{j,\,m-1})|.
\end{equation*}
Thus,
\begin{multline}
\sup_{u\in [0,\, t]}|I_3(u)|\leq \sum_{j=0}^\kappa \max_{1\leq m\leq 2^j}\,|I_3(v_{j,\,m})-I_3(v_{j,\,m-1})|\\
+\max_{0\leq j\leq 2^\kappa-1}\sup_{z\in [0,\,v_{\kappa,\,j+1}-v_{\kappa,\, j}]}|I_3(v_{\kappa,\,j}+z)-I_3(v_{\kappa,\, j})|\quad\text{a.s.}
\label{eq:inter15}
\end{multline}

We first show that, for all $\varepsilon>0$,
\begin{equation}\label{eq:inter8}
\lim_{t\to\infty}\mmp\Big\{\sum_{j=0}^{\kappa}
\max_{1\leq m\leq 2^j}\,|I_3(v_{j,\,m})-I_3(v_{j,\,m-1})|>\varepsilon t^{1/2}\Big\}=0.
\end{equation}
Let $s\in\mn$. To prove
\eqref{eq:inter8} we have to provide an appropriate upper bound for $\me (I_3(u)-I_3(v))^{2s}$ for $u>v>0$. Observe that $I_3(u)-I_3(v)$ is equal to the a.s.\ limit $\lim_{j\to\infty} R(j,u,v)$, where $(R(j,u,v), \mathcal{G}_j)_{j\geq 0}$ is a martingale defined by $R(0,u,v):= 0$, $$R(j, u,v):=\sum_{k=0}^{j-1}\big((\1_{\{S_k+\eta_{k+1}\leq u\}}-F(u-S_k))\1_{\{0<S_k\leq u\}}-(\1_{\{S_k+\eta_{k+1}\leq v\}}-F(v-S_k))\1_{\{0<S_k\leq v\}}\big),$$ and, as before, $\mathcal{G}_0$ denotes the trivial $\sigma$-algebra and, for $j\in\mn$, $\mathcal{G}_j$ denotes the $\sigma$-algebra generated by $(\xi_k,\eta_k)_{1\leq k\leq j}$. By the Burkholder–Davis–Gundy inequality, see, for instance, Theorem 11.3.2 in \cite{Chow+Teicher:2003},
\begin{multline*}
\me [(I_3(u)-I_3(v))^{2s}]\\\leq C\Big(\me \Big[\Big(\sum_{k\geq 0}\me\big((R(k+1, u,v)-R(k, u,v))^2|\mathcal{G}_k\big)\Big)^s\Big]+\sum_{k\geq 0}\me \big[(R(k+1, u,v)-R(k, u,v))^{2s}\big]\Big)\\=C\Big(\me \Big[\Big(\sum_{k\geq 0}F(u-S_k)(1-F(u-S_k))\1_{\{v<S_k\leq u\}}\\+\sum_{k\geq 0}(F(u-S_k)-F(v-S_k))(1-F(u-S_k)+F(v-S_k))\1_{\{0<S_k\leq v\}}\Big)^s\Big]\\+ \sum_{k\geq 0}\me\big[\big((\1_{\{S_k+\eta_{k+1}\leq u\}}-F(u-S_k))\1_{\{0<S_k\leq u\}}-(\1_{\{S_k+\eta_{k+1}\leq v\}}-F(v-S_k))\1_{\{0<S_k\leq v\}}\big)^{2s}\big]\Big)\\=: C(A(u,v)+B(u,v))
\end{multline*}
for a positive constant $C$. In what follows $C_1$, $C_2,\ldots$ denote positive constants whose values are of no importance. Further,
\begin{multline*}
A(u,v)=\me \Big[\Big(\int_{(v,\,u]}F(u-y)(1-F(u-y)){\rm d}N^\ast(y)\\+\int_{(0,\,v]}(F(u-y)-F(v-y))(1-F(u-y)+F(v-y)){\rm d}N^\ast(y)\Big)^s\Big]\\\leq 2^{s-1}\Big(\me \Big[\Big(\int_{(v,\,u]}F(u-y)(1-F(u-y)){\rm d}N^\ast(y)\Big)^s\Big]\\+\me\Big[\Big(\int_{(0,\,v]}(F(u-y)-F(v-y))(1-F(u-y)+F(v-y)){\rm d}N^\ast(y)\Big)^s\Big]\Big)\\\leq 2^{s-1}\Big(\me\Big[\Big(\int_{(0,\,u]}(1-F(u-y))\1_{[0,\,u-v)}(u-y){\rm d}N^\ast(y)\Big)^s\Big]+\me\Big[\Big(\int_{(0,\,v]}(F(u-y)-F(v-y)){\rm d}N^\ast(y)\Big)^s\Big]\Big)\\=:2^{s-1}(A_1(u,v)+A_2(u,v)).
\end{multline*}
Put $\gamma_s:=\me\Big[\Big(\sum_{k\geq 0}\1_{\{S_k\in (-1,1)\}}\Big)^s\Big]$ and $g(t):=\sum_{n=0}^{\lceil t\rceil}(1-F(n))$ for $t\geq 0$, where $x\mapsto \lceil x\rceil$ is the ceiling function. Using Lemma \ref{lem:intsum} with $t=u$ and $f(y)=(1-F(y))\1_{[0,\,u-v)}(y)$ and then with $t=v$ and $f(y)=F(u-v+y)-F(y)$ we infer $$A_1(u,v)\leq \gamma_s \Big(\sum_{n=0}^{\lfloor u\rfloor}\sup_{y\in [n,\,n+1)}\,(1-F(y))\1_{[0,\,u-v)}(y)\Big)^s\leq \gamma_s (g(u-v))^s$$ and
\begin{multline*}
A_2(u,v)\leq \gamma_s \Big(\sum_{n=0}^{\lfloor v\rfloor}\sup_{y\in [n,\,n+1)}\,(F(u-v+y)-F(y))\Big)^s\leq \gamma_s \Big(\sum_{n=0}^{\lfloor v\rfloor}(F(\lceil u-v\rceil+n+1)-F(n))\Big)^s\\=\gamma_s \Big(\sum_{n=0}^{\lceil u-v\rceil}(F(\lfloor v\rfloor+1+n)-F(n))\Big)^s\leq \gamma_s (g(u-v))^s.
\end{multline*}
Also,
\begin{multline*}
B(u,v)\leq \sum_{k\geq 0}\me\big[\big((\1_{\{v<S_k+\eta_{k+1}\leq u\}}-F(u-S_k)+F(v-S_k))\1_{\{0<S_k\leq v\}}\\+(\1_{\{S_k+\eta_{k+1}\leq u\}}-F(u-S_k))\1_{\{v<S_k\leq u\}}\big)^2\big]\leq 2 \gamma_1 g(u-v)
\leq 2\gamma_1 (g(u-v))^s
\end{multline*}
whenever $g(u-v)\geq 1$, and thereupon
\begin{equation}\label{eq:mom}
\me [(I_3(u)-I_3(v))^{2s}]\leq C_1 (g(u-v))^s
\end{equation}
whenever $g(u-v)\geq 1$.

Observe that $v_{j,\,m}-v_{j,\,m-1}=2^{-j}t$. Given $\delta>0$ $C_1 (g(2^{-j}t))^s \leq \delta 2^{-js}t^s$ for nonnegative integer $j\leq \kappa(t)$ and large $t$. Invoking \eqref{eq:mom} we then obtain, for nonnegative integer $j\leq \kappa(t)$ and large $t$,
\begin{equation}\label{eq:inter17}
\me [(I_3(v_{j,\,m})-I_3(v_{j,\,m-1}))^{2s}]\leq C_1 (g(2^{-j}t))^s\leq \delta 2^{-js} t^s
\end{equation}
and thereupon
\begin{multline*}
\me \big[\max_{1\leq m\leq 2^j}\,(I_3(v_{j,\,m})-I_3(v_{j,\,m-1}))^{2s}\big]\leq \sum_{m=1}^{2^j}\me\big[(I_3(v_{j,\,m})-I_3(v_{j,\,m-1}))^{2s}\big]\leq \delta 2^{-j(s-1)}t^s.
\end{multline*}
By the triangle inequality for the $L_{2s}$-norm, with integer $s\geq 2$,
\begin{multline*}
\me \Big[\Big(\sum_{j=0}^\kappa \max_{1\leq m\leq 2^j}\,|I_3(v_{j,\,m})-I_3(v_{j,\,m-1})|\Big)^{2s}\Big]\leq \Big(\sum_{j=0}^\kappa \big(\me \big[\max_{1\leq m\leq 2^j}\,(I_3(v_{j,\,m})-I_3(v_{j,\,m-1}))^{2s}\big]\big)^{1/(2s)}\Big)^{2s}\\\leq \delta t^s \big(\sum_{j\geq 0}2^{-j(s-1)/(2s)}\big)^{2s}=:C_2\delta t^s.
\end{multline*}
By Markov's inequality, for large $t$, $$\mmp\Big\{\sum_{j=0}^\kappa \max_{1\leq m\leq 2^j}\,|I_3(v_{j,\,m})-I_3(v_{j,\,m-1})|>\varepsilon t^{1/2}\Big\}\leq C_2\delta \varepsilon^{-2s}.$$ Letting $\delta\to 0+$ we arrive at
\eqref{eq:inter8}.

Now we pass to the analysis of the second summand in \eqref{eq:inter15}. Put $M(t):=\int_{(0,\,t]}F(t-y){\rm d}N^\ast(y)$ for $t\geq 0$. Using the equality $I_3(t)=N(t)-M(t)$ and a.s. monotonicity of $N$ and $M$ we infer
\begin{multline*}
\sup_{z\in [0,\,v_{\kappa,\,j+1}-v_{\kappa,\,j}]}|I_3(v_{\kappa,\,j}+z)-I_3(v_{\kappa,\,j})|\\\leq \sup_{z\in [0,\,v_{\kappa,\,j+1}-v_{\kappa,\,j}]}(N(v_{\kappa,\,j}+z)-N(v_{\kappa,\,j}))+\sup_{z\in [0,\,v_{\kappa,\,j+1}-v_{\kappa,\,j}]}\,(M(v_{\kappa,\,j}+z)-M(v_{\kappa,\,j}))\\=(N(v_{\kappa,\,j+1})-N(v_{\kappa,\,j}))+ (M(v_{\kappa,\,j+1})-M(v_{\kappa,\,j})).
\end{multline*}
Observe that
\begin{multline*}
\max_{0\leq j\leq 2^\kappa-1}\,(N(v_{\kappa,\,j+1})-N(v_{\kappa,\,j}))\leq \max_{0\leq j\leq 2^\kappa-1}\,|I_3(v_{\kappa,\,j+1})-I_3(v_{\kappa,\,j})|\\+\max_{0\leq j\leq 2^\kappa-1}\,(M(v_{\kappa,\,j+1})-M(v_{\kappa,\,j})).
\end{multline*}
Hence, it is enough to prove that, for all $\varepsilon>0$,
\begin{equation}\label{eq:inter9}
\lim_{t\to\infty} \mmp\big\{\max_{0\leq j\leq 2^\kappa-1}\,(M(v_{\kappa,\,j+1})-M(v_{\kappa,\,j}))>\varepsilon t^{1/2}\big\}=0
\end{equation}
and
\begin{equation}\label{eq:inter10}
\lim_{t\to\infty} \mmp\big\{\max_{0\leq j\leq 2^\kappa-1}\,|I_3(v_{\kappa,\,j+1})-I_3(v_{\kappa,\,j})|>\varepsilon t^{1/2}\big\}=0.
\end{equation}
Arguing as above we conclude that, for $u>v>0$,
\begin{multline*}
\me [(M(u)-M(v))^s]=\me\Big[\Big(\int_{(v,\,u]}F(u-y){\rm d}N^\ast(y)+\int_{(0,\,v]}(F(u-y)-F(v-y)){\rm d}N^\ast(y)\Big)^s\Big]\\\leq 2^{s-1}\gamma_s (\lceil u-v\rceil+1)^s.
\end{multline*}
As a consequence, for nonnegative integer $j\leq \kappa(t)$ and large $t$, $$\me [(M(v_{\kappa,\,j+1})-M(v_{\kappa,\,j}))^s]
\leq C_3 2^{-\kappa s} t^s.$$ By Markov's inequality and our choice of $\kappa$,
\begin{multline*}
\mmp\big\{\max_{0\leq j\leq 2^\kappa-1}\,(M(v_{\kappa,\,j+1})-M(v_{\kappa,\,j}))>\varepsilon t^{1/2}\big\}\leq C_3\varepsilon^{-s}2^{-\kappa(s-1)}t^{s/2}\leq C_3\varepsilon^{-s} 2^{s-1}t^{3/4-s/4}.
\end{multline*}
Hence, \eqref{eq:inter9} follows upon choosing $s=4$, say. To prove
\eqref{eq:inter10}, we invoke \eqref{eq:inter17} which enables us to conclude that
\begin{equation*}
\mmp\big\{\max_{0\leq j\leq 2^\kappa-1}\,|I_3(v_{\kappa,\,j+1})-I_3(v_{\kappa,\,j})|>\varepsilon t^{1/2}\big\}
\leq C_2\varepsilon^{-2s}\delta^s 2^{-\kappa(s-1)}.
\end{equation*}
Choosing $s=2$ and letting $t\to\infty$ we arrive at
\eqref{eq:inter10}.

\noindent {\sc Analysis of $I_4$.} It is known (see, for instance, Proposition A.1 in \cite{Iksanov+Kondratenko:2021}) that
\begin{equation}\label{eq:N*}
\Big(\frac{N^*(ut) - \mu^{-1}ut}{(\sigma^2\mu^{-3}t)^{1/2}}\Big)_{u\geq0}~\Longrightarrow~ \big(B(u)\big)_{u\geq0},\quad t\to\infty
\end{equation}
in the $J_1$-topology on $D$. An application of Lemma \ref{lem:Iks}, with $h=F$, $M(t)=N^\ast(t)-N^\ast(0)$, $a=1$, $b=\mu$, $c=\sigma \mu^{-3/2}$, then yields $$\Big(\frac{\sum_{k\geq 0}F(ut-S_k)\1_{\{0<S_k\leq ut\}}-\mu^{-1}\int_0^{ut}F(y){\rm d}y}{(\sigma^2 \mu^{-3}t)^{1/2}}\Big)_{u\geq 0}~\Longrightarrow~(B(u))_{u\geq 0},\quad t\to\infty$$ in the $J_1$-topology on $D$.

The proof of Theorem \ref{thm5} is complete.
\end{proof}

\begin{rem}\label{rem:as}
As was announced in the proof of Theorem \ref{thm5}, we explain here that the relation $\lim_{t\to\infty} t^{-1/2}\sup_{u\in [0,\,1]}\,|I_3(ut)|=0$ a.s.\ may fail to hold. Assume that $\mmp\{\xi=c\}=1$ for some $c>0$ and $1-F(t)=\mmp\{\eta>t\}\sim (\log\log t)^{-1/2}$ as $t\to\infty$. It can be checked, with some efforts, that Theorem 1.6 in \cite{Buraczewski+Iksanov+Kotelnikova:2025} applies. By that theorem, $$\limsup_{t\to\infty}\frac{I_3(t)}{t^{1/2}(\log\log t)^{1/4}}=\Big(\frac{2}{c}\Big)^{1/2}\quad\text{a.s.}$$ and thereupon $\lim_{t\to\infty} t^{-1/2}\sup_{u\in [0,\,1]}\,|I_3(ut)|=\infty$ a.s.
\end{rem}

\noindent \textbf{Acknowledgment.} We thank the two anonymous referees for numerous very useful comments which helped to significantly improve the presentation and correct several our blunders. The work of A. Iksanov was supported by the National Research Foundation of Ukraine (project 2023.03/0059 ‘Contribution to modern theory of random series’).

\end{document}